\documentclass[12
pt]{amsart}
\usepackage{latexsym}
\usepackage{amssymb}
\usepackage{amsmath}

\usepackage{enumerate}
\usepackage{xcolor}

\usepackage{amsfonts}

\newtheorem{theorem}{Theorem}[section]
\newtheorem{lemma}[theorem]{Lemma}
\newtheorem{proposition}[theorem]{Proposition}
\newtheorem{corollary}[theorem]{Corollary}
\newtheorem{definition}[theorem]{Definition}

\newtheorem{remark}[theorem]{Remark}

\newcommand\Ball{\mathop{\rm Ball}}

\newcommand{\cl}[1]{\mathcal{#1}}
\newcommand{\bb}[1]{\mathbb{#1}}

\newcommand{\sca}[1]{\left\langle#1\right\rangle} 

\begin{document}

\title[Hyperreflexivity and Similarity]{Hyperreflexivity of von Neumann algebras and similarity of finitely generated $C^*$-algebras.}

\author{G. K. Eleftherakis, V. I. Paulsen}
\date{}
\address{}

\email{}

\begin{abstract}Let $\cl A$ be a $C^*$-algebra. We say that $\cl A$ satisfies the SP if every bounded homomorphism $\cl A\to B(K)$, with  $K$ a Hilbert space, is similar to a $*$-homomorphism. We introduce three hypotheses that relate to extending hyperreflexive algebras by projections. We prove that our third hypothesis is equivalent to every finitely generated C*-algebra satisfying the SP.

We show that to prove that every von Neumann algebra is hyperreflexive it is enough to show that when one extends a hyperreflexive algebra by a single projection it remains hyperreflexive.

 \end{abstract}

\maketitle

\section{Preliminaries}

This paper focuses on the longstanding {\it Kadison Similarity Problem}: whether every C*-algebra $\cl A$ possesses the similarity property (SP). Specifically, the question of whether every bounded homomorphism from a C*-algebra into $B(H),$ where $H$ is a Hilbert space, is similar to a *-homomorphism.

We establish conditions under which all C*-algebras generated by finitely many elements satisfy the SP. Parallel results hold for the weak similarity property (WSP) in the context of finitely generated von Neumann algebras. It should be noted that another major open problem in the theory concerns whether von Neumann algebras acting on a separable Hilbert space admit a single generator. If the answer to the single generator problem is affirmative, then our results would give necessary and sufficient conditions for the WSP problem to also have an affirmative answer.

The main tool in our proof is the concept of hyperreflexivity. It is well known that the similarity property holds for all C*-algebras if and only if all von Neumann algebras are hyperreflexive. In \cite{ep}, it was shown that a given von Neumann algebra $\cl A$  satisfies the WSP if and only if its commutant $\cl A'$ is completely hyperreflexive, meaning that the algebra $\cl A'\otimes B(\ell^2(I))$ is hyperreflexive for all cardinals $I.$

Using this framework, we study three ways of extending an algebra by a projection. We prove that the first way(EPH.1) preserves the property of being hyperreflexive if and only if every von Neumann algebra is hyperreflexive, and consequently that the similarity problem has an affirmative answer. Thus, the similarity problem is equivalent to determining if hyperreflexivity is preserved by a certain one-step extension property.  We prove that the third way to extend(EPH.3) preserves hyperreflexivity if and only if every finitely generated von Neumann algebra is hyperreflexive. As a consequence, we show that if every separably acting von Neumann algebra is singly generated(another famous problem of Kadison), then these two hypotheses are equivalent. 

Below, we list some related works that have been published on this topic, 
\cite{chris, ep, haag, kad, ki, lm, pap1, pap2, pisier}.

In what follows the symbol $B(H)$ is used for the space of bounded linear operators acting on the Hilbert space $H.$ 

\begin{definition} Let $\cl X\subseteq B(H)$ be an operator space and $T \in B(H)$. We set
$$r_{\cl X}(T)=\sup\{|\sca{T\xi, \eta}|: \|\xi\|=\|\eta\|=1,\sca{X\xi, \eta}=0, \forall X\in \cl X\}.$$
\end{definition}

\begin{definition} Let $\cl X\subseteq B(H)$ be a reflexive space. We call $\cl X$ hyperreflexive if there exists $k>0$ such that 
$$d(T, \cl X)\le k r_{\cl X}(T), \forall T\in B(H).$$
\end{definition}

\begin{definition}Let $\cl X\subseteq B(H)$ be a reflexive space. We define 
$$k(\cl X)=\sup_{T\notin \cl X}\frac{d(T, \cl X)}{r_{\cl X}(T)}$$ to be the hyperreflexivity constant of $\cl X.$
\end{definition}

Clearly, $\cl X$ is hyperreflexive if and only if $k(\cl X)<\infty .$

\begin{remark} If $\cl X$ is a dual  operator space and $I$ is a set, then $M_I(\cl X)$ is the space of $I\times I$ matrices with entries in $\cl X$ whose finite submatrices have uniformly bounded norm. The space $M_I(\cl X)$  is also a dual operator space. One can also identify $M_I(\cl X)$ with the space $\cl X\bar \otimes B(l^2(I)). $
\end{remark}

\begin{definition} Following \cite{ep}, we call a reflexive space $\cl X \subseteq B(H)$, completely hyperreflexive, provided $M_I(\cl X)$ is hyperreflexive for every set $I$.
\end{definition}

In the next section we show that this definition of completely hyperreflexive agrees with the definition given earlier by  Davidson and Levene \cite{dl}.

\begin{theorem}\label{hy}\cite[Theorem 4.6]{ele} Let $\cl A, \cl B$ be von Neumann algebras and $\pi: \cl A\to \cl B$ a $*$-isomorphism. Then there exists a set $J_0$ such  that $$k(M_J(\cl A'))=k(M_J(B')$$ for all $J\supseteq J_0.$

\end{theorem}

\begin{corollary}\label{hyp} Let $\cl A, \cl B$ be von Neumann algebras and $\pi: \cl A\to \cl B$ be an onto  $w^*$-continuous $*$-homomorphism. Then there exists a set $J_0$ such  that $$k(M_J(\cl B'))\le k(M_J(\cl A'))$$ for all $J\supseteq J_0.$
\end{corollary}
\begin{proof} There exists a central projection $P\in \cl A$  such that the map $$\cl AP\to \cl B,\;\;AP\to \pi(A)$$ is a $*$-isomorphism. By the above Theorem there exists a set $J_0$ such  that $$k(M_J(\cl B'))= k(M_J(\cl A'P))\le  k(M_J(\cl A'))$$ for all $J\supseteq J_0.$

\end{proof}

\begin{definition} Let $\cl A$ be a $C^*$-algebra. We say that $\cl A$ satisfies the SP(similarity property)  if for all bounded homomorphisms $u: \cl A\to B(H)$ there exists an invertible operator $S\in B(H)$ such that the map $\pi(\cdot)=S^{-1}u(\cdot)S$ is a $*$-homomorphism of $\cl A.$
\end{definition}

\begin{definition} Let $\cl A$ be a von Neumann algebra. We say that $\cl A$ satisfies the WSP if for all $w^*$-continuous bounded homomorphisms $u: \cl A\to B(H)$ there exists an invertible operator $S\in B(H)$ such that the map $\pi(\cdot)=S^{-1}u(\cdot)S$ is a $*$-homomorphism of $\cl A.$
\end{definition}

\begin{remark}\cite{ep} Let $\cl A$ be a $C^*$-algebra.Then $\cl A$ satisfies the SP iff $\cl A^{**}$ satisfies the WSP.
\end{remark}

\begin{theorem}\cite{ep} \label{pep} Let $\Omega \subseteq B(H)$ be a von Neumann algebra. Then $\Omega$ satisfies the WSP if and only if for all sets $J$ the algebra 
$M_J(\Omega^\prime)$ is hyperreflexive.
\end{theorem}

\begin{remark}\label{ep} If $\Omega$ is a von Neumann algebra satisfying the WSP then there exists $0<M<\infty$ such that 
$k(M_J(\Omega^\prime))\le M$ for all sets $J.$
\end{remark}

We say that the $C^*$-algebra (resp. von Neumann algebra) $\cl A$ is finitely generated if there exists a finite set $F\subseteq \cl A$ which generates $\cl A$ as a $C^*$-algebra (resp. von Neumann algebra).

We introduce the following hypotheses:

EPH.1 For every hyperreflexive von Neumann algebra $\cl   A$ acting on the Hilbert space $H$ and every projection $Q\in B(H)$ the algebra $\cl A\vee \{Q\}^{''}$ is   hyperreflexive.

EPH.2  For every completely hyperreflexive von Neumann algebra $\cl A$ acting on the Hilbert space $H$ and every projection $Q\in B(H)$ the algebra  $\cl A\wedge \{Q\}'$ is completely 
 hyperreflexive.

EPH.3 For every $Q_1, Q_2,...,Q_n, n\in \bb N$ projections the algebra $$\{Q_1 \}'\wedge ...\wedge\{Q_n\}'$$ is hyperreflexive.

We prove that EPH.1 is equivalent to the fact that every von Neumann algebra is hyperreflexive and consequently, that all $C^*$-algebras satisfy SP. We thank the anonymous referee of an earlier version of this paper for essentially this observation. We prove that EPH.2 implies EPH.3 and EPH.3 is equivalent to the statement that every finitely generated $C^*$-algebra satisfies SP.


\section{Similarity and projections}

In this section, we prove the interesting property that the Kadison problem has a positive solution if and only if the extension of every 
hyperreflexive von Neumann  algebra by a projection is also a hyperreflective algebra.

\begin{lemma}\label{union} Let $\cl X_i, i\in I$ be hyperreflexive operator spaces acting on the Hilbert space $H,$ such that 
$\forall i, j, \cl X_i\subseteq\cl  X_j, \;\;or\;\;\cl X_j\subseteq \cl X_i.$ We denote $\cl X=\overline{\cup_i\cl X_i}^{w^*}$ and we assume that  $\kappa=\sup_ik(\cl X_i)<\infty .$ Then $\cl X$ is hyperreflexive and $k(\cl X)\le \kappa .$
\end{lemma}
\begin{proof} We denote  $i\le j$ if and only if  $\cl X_i\subseteq \cl X_j.$ Let $T\in B(H),$ then $$d(T, \cl X)\le d(T, \cl X_i)\le \kappa r_{\cl X_i}(T).$$ Since $r_{\cl X_j}(T)\le r_{\cl X_i}(T)$ for all $j\ge i$ we have that 
$$d(T, \cl X)\le \kappa \lim_ir_{\cl X_i}(T)=\kappa \inf _ir_{\cl X_i}(T).$$
Choose $\epsilon_i>0, i\in I$ such that $\lim_i\epsilon_i=0.$

We can find unit vectors $\xi_i, \eta_i$ such that 
$$\sca{X\xi_i ,\eta_i}=0, \forall X\in \cl X_i$$ and $$r_{\cl X_i}(T)-\epsilon_i\le |\sca{T\xi_i, \eta_i}|\le r_{\cl X_j}(T)$$ forall $j\le i.$
Thus $$\lim_i(r_{\cl X_i}(T)-\epsilon_i)\le \lim_i |\sca{T\xi_i, \eta_i}| \le \lim_jr{\cl X_j(T)} \Rightarrow $$
 $$\lim_i|\sca{T\xi_i, \eta_i}|=\inf_ir_{\cl X_i}(T).$$
We conclude that 
$$d(T, \cl X)\le \kappa \lim_i |\sca{T\xi_i, \eta_i}|.$$
Passing to a subnet if it is neseccary we may assume that $$weak^*-\lim_i \xi_i\otimes \eta_i^*=\xi\otimes \eta^*.$$
Since $\sca{X\xi, \eta }=0$ for all $X\in \cl X_i$ and for all $i$ we have that  $\sca{X\xi, \eta }=0$ for all $X\in \cl X.$
Thus $$d(T, \cl X)\le \kappa \sca{T\xi, \eta}\le \kappa r_{\cl X}(T).$$

\end{proof}
Given a hyperreflexive space $\cl X$, \cite{dl} set
\[ k_{cb}(\cl X) = \sup_{n \in \bb N} k(M_n(\cl X)),\]
where we identify $M_n(\cl X) = M_{\{1,...,n\}}(\cl X)$ and $\bb N$ denotes the natural numbers.

\begin{lemma}\label{union2} Let $\cl X$ be a dual operator space such that $k(M_{\bb N} (\cl X))<\infty .$ 
Then  $k_{cb}(\cl X) = k(M_{\bb N}(\cl X))=k(M_{I}(\cl X)),$ for any infinite set $I$.
\end{lemma}
\begin{proof} 
Since $M_{\bb N}(\cl X) = \overline{\cup M_n(\cl X)}^{w^*}$, by Lemma~\ref{union},
\[ k(M_{\bb N}(\cl X) \le \sup_n k(M_n(\cl X)) \le k(M_{\bb N}(\cl X),\]
and the first equality follows.

Conversely, 
let $\cl X \subseteq B(\cl H)$ and let $T= (T_{i,j}) \in M_I(B(\cl H))$.  Pick $X_n \in M_I(\cl X)$ such that
\[ d(T, M_I(\cl X) = \lim_n \|T - X_n \|\]
and pick finite subsets $E_n \subseteq I$ such that
\[ \|T - X_n \| \le \|P_n TP_n - P_nX_n P_n \| + 1/n,\]
where $P_n$ is the projection onto the set $E_n$.
If we let $E= \cup_n E_n$ then $E$ is a countable set and setting $R$ be the compression of $T$ to $E$ we have that
\[ d(T, M_I(\cl X)) = d(R, M_E(\cl X)) \le k_{cb}(\cl X) r_E(R, M_E(\cl X)) \le k_{cb}(\cl X) r_I(T, M_I(\cl X)).\]
Thus,  $M_I(\cl X)$ is hyperreflexive with constant at most $k_{cb}(\cl X)$, and the result follows.
\end{proof}

The above Lemma improves the main result of \cite{ep}, see \ref{ep}:

\begin{theorem}\label{new}  The von Neumann algebra $\cl A$ satisfies the WSP if and only if $k(M_{\bb N} (\cl A'))<\infty .$
\end{theorem}

\begin{theorem} The EPH.1 hypothesis is equivalent to every  von Neumann algebra being hyperreflexive, and consequently, to every $C^*$-algebra satisfying the SP. 
\end{theorem}
\begin{proof} By \cite{ep}, to deduce that every $C^*$-algebra satisfies the SP, it suffices to prove that every  separably acting von Neumann algebra $\cl A$ satisfies WSP.

Given any von Neumann algebra $\cl A$, the algebra $\cl A'\bar \otimes \bb CI_{\ell^2}$ is hyperreflexive. By \cite{DAVIS}, there exist projections $\{P_1, P_2, P_3\}$ such that 
$$B(\ell^2)=\{P_1, P_2, P_3\}^{''}.$$

We have that
$$\cl A'\bar \otimes B(\ell^2)=(\cl A'\bar \otimes \bb CI_{\ell^2})\vee (\{I\otimes P_1\}^{''}) \vee (\{I\otimes P_2\}^{''})) \vee (\{I\otimes P_3\}^{''})). $$
Thus by EPH.1 the algebra $\cl A'\bar \otimes B(\ell^2)$ is hyperreflexive. Therefore by  \ref{new} $\cl A$ satisfies WSP.
Since $\cl A$ was arbitrary, every $C^*$-algebra satisfies the SP, and so, equivalently, every von Neumann algebra is hyperreflexive.

Conversely, if every $C^*$-algebra satisfies the SP, then every von Neumann algebra satisfies the WSP and so by Theorem  \ref{new}, every von Neumann algebra is hyperreflexive and so EPH.1 trivially holds.
 \end{proof}

\section{Similarity and Finitely Generated Algebras}

\begin{theorem}\label{xxx}  EPH.2 assumption implies EPH.3
\end{theorem}
\begin{proof}  Let $Q_i, 1\le i\le n$ be projections. For a cardinal $I$ using induction on $n$ we have that  $$\{Q_1\otimes  I _{\ell^2(I)} \}'\wedge ...\wedge\{Q_n\otimes I _{\ell^2(I)}\}'= M_I(\{Q_1\}'\wedge ...\wedge \{Q_n\}')$$
and this algebra is hyperreflexive, because of  assumption EPH.2. 
\end{proof}

\begin{lemma}\label{main} Let $\cl A$ be a $C^*$-algebra generated by a finite set of projections. Under the EPH.3 hypothesis, $\cl A$ satisfies SP.
\end{lemma}
\begin{proof} Assume that $\cl A$ is generated by the projections $P_1,...,P_n.$ We may assume that 
$$\cl A=C^*(P_1,...,P_n)\subseteq W^*(P_1,...,P_n)=\cl A^{**}=\Omega\subseteq B(H).$$
It suffices to prove that $\Omega$ satisfies WSP.

Let $J$ be a cardinal. From Theorem~\ref{pep} it suffices to prove that the algebra $M_J(\Omega^\prime)$ is hyperreflexive.
We have that 
$$M_J(\Omega^\prime)=\cap_{i=1}^n(\{P_i\}^\prime\bar \otimes B(l^2(J))).$$ 
By Lemma~\ref{xxx} this  algebra  
is hyperreflexive.

\end{proof}

\begin{lemma} Let $\cl A$ be a $C^*$-algebra and $n \in \bb N$.  Then $\cl A$ satisfies SP if and only if $M_n(\cl A)$ satisfies SP.
\end{lemma}
\begin{proof} Assume that $M_n(\cl A)$ satisfies SP and let $u: \cl A \to B(H)$ be a bounded homomorphism.  Then $id_n \otimes u: M_n(\cl A) \to B( \bb C^n \otimes H)$ is a bounded homomorphism and so it is similar to a *-homomorphism, and hence completely bounded.  But this implies that $u$ is completely bounded and hence similar to a *-homomorphism.

Conversely, if $\cl A$ has SP and $\pi: M_n(\cl A) \to B(H)$ is a bounded homomorphism, then the restriction of $\pi$ to the $C^*$-algebra $M_n \simeq M_n \otimes 1_{\cl A} \subseteq M_n(\cl A)$ is a bounded homomorphism and since $M_n$ has SP, it is an invertible $S$ so that $\rho = S^{-1} \pi S$ is a *-homomorphism when restricted to $M_n$.  This allows us to decompose $\cl H \simeq \bb C^n \otimes  K$ in so that $\rho= id_n \otimes u: M_n(\cl A) \to B(\bb C^n \otimes K)$.  Since $\cl A$ has SP, there is an invertible $R$ on $K$ so that $R^{-1} u R: \cl A \to B(K)$ is a *-homomorphism. Conjugating $\pi$ by $(id_n \otimes R)S$ converts $\pi$ into a *-homomorphism.(Alternatively, since $\cl A$ has SP, $u$ is completely bounded and hence $id_n \otimes u$ and $\pi$ are comletely bounded.)
\end{proof}

\begin{lemma}\label{v} Let $\cl A$ be a $C^*$-algebra that is finitely generated.  Then $M_2(\cl A)$ is generated by finitely many projections.
\end{lemma}
\begin{proof} By scaling we may assume that the generators are contractions. By taking real and imaginary parts of the generators, we may assume that they are self-adjoint elements of norm at most 1.
So let $\{ H_1,..., H_n \}$ be self-adjoint contractions that generate $\cl A.$
Then $U_j = e^{iH_j}, \, 1 \le j \le n$ are unitaries that generate $\cl A$.

It is easily checked that 
\[ \begin{pmatrix} I_{\cl A} & 0 \\ 0 & 0 \end{pmatrix}, 1/2 \begin{pmatrix} I_{\cl A} & I_{\cl A} \\ I_{\cl A} & I_{\cl A} \end{pmatrix}, 1/2 \begin{pmatrix} I_{\cl A} & U_j \\U_j^* & I_{\cl A} \end{pmatrix}, 1 \le j \le n,\]
are $n+2$ projections that generate $M_2(\cl A)$.
\end{proof}

\begin{theorem}\label{11} Let $\cl A$ be a $C^*$-algebra that is finitely generated.  Then under EPH.3 hypothesis, $\cl A$ satisfies SP.
\end{theorem}

\begin{corollary}  Under EPH.3 hypothesis, the full and reduced $C^*$-algebras of any finitely generated discrete group satisfy SP.
\end{corollary}

\begin{remark}\em{ If the SP holds for every separable C*-algebra, then it holds for every C*-algebra. One way to see this is to assume that the SP fails for $\cl A$, then by \cite[Theorem 9.1]{pau} there exists a bounded homomorphism $\rho: \cl A \to B(\cl H)$ that is not completely bounded. Hence there are elements $(a_{i,j,n}) \in M_n(\cl A)$ of norm one with $\| ( \rho(a_{i,j,n}) \| \to \infty$. If we let $\cl B$ be the separable C*-subalgebra of $\cl A$ generated by the set $\{ a_{i,j,n}: 1 \le i,j \le n, \, n \in \bb N \}$, then the restriction of $\rho$ to $\cl B$ is not completely bounded and so $\cl B$ fails the SP.

Furthermore, since every separable C*-algebra is a quotient of $C^*(\bb F_{\infty})$, the SP holds for every C*-algebra  if and only if it holds for $C^*(\bb F_{\infty})$, where $\bb F_{\infty}$ denotes the free group on countably infinitely many generators. By the above results, SP does hold for $C^*(\bb F_n)$ for every $n \in \bb N$.}
\end{remark} 

For the following result it helps to recall that if $\cl A$ is a non-unital C*-algebra and $\rho: \cl A \to B(\cl H)$ is a bounded homomorphism, then $\rho$ extends to a homomorphism of the unitization of $\cl A$ with the same norm.

\begin{proposition} Under the EPH.3 hypothesis, for each $n \in \bb N$ there is a function $f_n: [1, + \infty) \to [1, +\infty)$ such that if $\cl A$ is a C*-algebra with $n$ generators and $\rho: \cl A \to B(\cl H)$ is a homomorphism with $\| \rho \| \le c$, then $\| \rho \|_{cb} \le f_n(c)$.
\end{proposition}
\begin{proof} Assume that for some $n$ no such function exists. Then there exists $c \in (1, +\infty)$ and a sequence of  unital C*-algebras $\cl A_j$ each with $n$ generators and a sequence of homomorphisms, $\rho_j: \cl A_J \to B(\cl H_j)$ such that
$\| \rho_j \| \le t$ but $\sup_j \|\rho_j \|_{cb} = + \infty.$. Let $\cl B$ be the C*-algebra that is the $c_0$-direct sum of the $\cl A_j$'s.

If we define a map $\rho: \cl B \to B(\oplus_2 \cl H_j)$ to be the direct sum of the maps $\rho_j$, then $\rho$ is a homomorphism with $\|\rho \| \le c$ and $\| \rho \|_{cb} = + \infty$, i.e., $\rho$ is not completely bounded.

However, if we choose generating sets $\{ a_{k,j} : 1 \le k \le n \} \subseteq \cl A_j$ with $\| a_{k,j} \| \le 1$ and set
$b_k = \oplus_j \frac{a_{k,j}}{j}, \, 1 \le k \le n$ and $b_0 = \oplus_j \frac{ 1_{\cl A_j}}{j}$, then this is a set of generators for $\cl B$.

Thus, $\cl B$ is a finitely generated C*-algebra that fails SP.  This contradiction completes the proof.
\end{proof}

\begin{theorem} Under the EPH.3 hypothesis, for each $n \in \bb N$ there are constants $K_n$ and $d_n \in \bb N$ such that if $\cl A$ is a unital C*-algebra with $n$ generators and $\rho: \cl A \to B(\cl H)$ is a bounded unital homomorphism, then
\[ \| \rho \|_{cb} \le K_n \|\rho \|^{d_n}.\]
\end{theorem}
\begin{proof} By Pisier's theory of similarity and factorization degree \cite{pisier}(see in particular \cite[Theorem~19.11]{pau}), it is enough to know that for any fixed $c>1$ there is a number $M$ such that for every unital homomorphism $\rho: \cl A \to B(\cl H)$,  $\| \rho \| \le c \implies \|\rho \|_{cb} \le M$ to deduce that for any bounded unital homomorphism $\pi: \cl A \to B(\cl H)$, we have that
\[ \| \pi \|_{cb} \le K \|\pi \|^d,\]
for any $d$ satisfying $M < c^d(c-1)$ and $K = \frac{M(c^d -1)}{c^{d+1} - c^d - M}$.
Since we may use $M= f_n(c)$ for every C*-algebra with $n$ generators, the result follows.
\end{proof}
\begin{remark} By Pisier's remarkable theorem that the similarity degree and factorization degrees are identical \cite{pisier}, if the EPH.3 hypothesis holds, then we have that every C*-algebra with $n$ generators has factorization degree $(d_n, K_n)$ for the above constants.
\end{remark}

\section{Weak similarity and Finitely Generated von Neumann Algebras}

There is a long history of results concerning single generation for von Neumann algebras, another famous problem of Kadison. 
There is a fairly extensive literature on this problem. For a fairly recent survey see \cite{SS}.

The results of this section provide a small link between Kadison's two problems by showing that if every separably acting von Neumann algebra is singly generated, or even finitely generated, then EPH.3 is equivalent to the WSP property.

\begin{theorem}\label{maponto}Let $\cl A, \cl B$ be von Neumann algebras and let $\pi: \cl A\to \cl B$ be an onto  $w^*$-continuous $*$-homomorphism.
If $\cl A$ satisfies the WSP, then $\cl B$ also satisfies the WSP.
\end{theorem}
\begin{proof} 
Since $\cl A$ satisfies the WSP then, 
$$k(\cl A'\bar \otimes B(l^2))<\infty.$$
Corollary~\ref{hyp} implies that $$k(\cl B'\bar \otimes B(l^2))<\infty.$$ Theorem \ref{new} implies that $\cl B$ satisfies the WSP.
\end{proof}

\begin{theorem}\label{main2}  Every finitely generated C*-algebra satisfies the SP if and only if EPH.3 holds.
\end{theorem}
\begin{proof} It suffices to prove the converse of Theorem~\ref{11}.

Let $ Q_1,...,Q_n$ be projections in $B(H),$ let $\cl A= C^*(Q_1,...,Q_n)$ and let $\Omega= \cl A^{''} = VN(Q_1,...,Q_n)$ so that $\Omega' = {Q_1}' \cap... \cap {Q_n}'.$

Since $\cl A$  satisfies SP, $\cl A^{**}$ satisfies WSP.  But there is a wk*-continuous *-homomorphism from $\cl A^{**}$ onto $\Omega$ so by Lemma \ref{xxx}, $\Omega$ satisfies WSP.

Hence by Theorem \ref{ep},  $\Omega'$  is completely hyperreflexive. Hence EPH.3 holds.
\end{proof}

\begin{lemma}\label{5} Let $\cl A$ be a von Neumann algebra weakly generated by a finite set of projections. Then under the EPH.3 hypothesis,  $\cl A$ satisfies the WSP.
\end{lemma}
\begin{proof} Assume that $\cl A$ is generated by the projections $P_1,...,P_n.$ . From Theorem \ref{new} it suffices to prove that the algebra $\cl A^\prime\bar \otimes B(l^2)$ is hyperreflexive.
We have that 
$$\cl A^\prime\bar \otimes B(l^2)=\cap_{i=1}^n(\{P_i\}^\prime\bar \otimes B(l^2))=(\cap_{i=1}^n(\{P_i\}^\prime)\bar \otimes B(l^2).$$ 
 
We have that 
$$\cl A^\prime\bar \otimes B(l^2)=\cap_{i=1}^n(\{P_i\}^\prime\bar \otimes B(l^2))=(\cap_{i=1}^n \{P_i\}^\prime )\bar \otimes B(l^2)).$$ By Lemma \ref{xxx} this algebra is 
 hyperreflexive.

\end{proof}

\begin{lemma}\label{6} Let $\cl A$ be a von Neumann algebra and $n\in \bb N.$ Then  $\cl A$ satisfies the WSP if and only if $M_n(\cl A)$ satisfies the WSP.
\end{lemma}
\begin{proof} Assume that $M_n(\cl A)$ satisfies WSP and let $u: \cl A \to B(H)$ be a $w^*$-continuous, bounded homomorphism.  Then $id_n \otimes u: M_n(\cl A) \to B( \bb C^n \otimes H)$ is a $w^*$-continuous bounded homomorphism and so it is similar to a *-homomorphism, and hence completely bounded.  But this implies that $u$ is completely bounded and hence similar to a *-homomorphism.

For the converse  direction we use Corollary 4.7 in \cite{ep}.
\end{proof}

\begin{lemma}\label{7} Let $\cl A$ be a von Neumann algebra weakly generated by a finite set.  Then  the algebra $M_2(\cl A)$ is weakly generated by a finite set of projections. 
\end{lemma}
\begin{proof} Let $\cl C$ be a selfadjoint finite set of 
 projections such that $\cl A=\cl C''.$ Define 
$\cl A_0=C^*(\cl C).$ Then $\cl A=\overline{\cl A_0}^{w^*}$ and $$M_2(\cl A)=\overline{M_2(\cl A_0)}^{w^*}.$$

By Lemma \ref{v} $M_2(\cl A_0)$ is finitely genarated by projections, thus $M_2(\cl A)$ is  weakly generated by a finite set of projections. 
\end{proof}

\begin{theorem}\label{8} Let $\cl A$ be a von Neumann algebra weakly generated by a finite set. Then  under the EPH.3 hypothesis, $\cl A$ satisfies the WSP. Conversely, if every von Neumann algebra that is weakly generated by a finite set satisfies the WSP, then EPH.3 holds.
\end{theorem}
\begin{proof} By Lemmas \ref{5}, \ref{6}, \ref{7} the algebra $\cl A$ satisfies the WSP.

Conversely, let $Q_1,..., Q_n \in B(H)$ be projections and let $\cl A$ be the von Neumann algebra generated by $Q_1,..., Q_n$. Then $\cl A$ satisfies the WSP and so $\cl A^{\prime} = \{Q_1\}^{\prime} \wedge \cdots \{ Q_n \}^{\prime}$ is hyperreflexive by \cite{ep}.
\end{proof}

\begin{corollary} If every von Neumann algebra acting on a separable Hilbert space is singly generated, then EPH.1 and EPH.3 are equivalent.
\end{corollary}
\begin{proof} We already know that EPH.1 implies EPH.3.  Assuming that EPH.3 and the single generator property hold, then we have that every von Neumann algebra on a separable Hilbert space satisfies WSP. From this it follows that every von Neumann algebra satisfies the WSP, \cite{ep}. Thus, every von Neumann algebra is hyperreflexive and so EPH.1 holds.
\end{proof}

\begin{corollary}   Under the EPH.3 hypothesis, the group von Neumann algebra $VN(G)$ of any finitely generated discrete group satisfies the WSP.
\end{corollary}
\begin{proof} The algebra $VN(G)$ is weakly finitely generated, so the claim is implied from Theorem \ref{8}.
\end{proof}

There are many von Neumann algebras that are known to be finitely generated and, hence, under the EPH.3 hypothesis, satisfy the WSP.

We now turn our attention to EPH.2, which lies between EPH.1 and EPH.3.

\begin{theorem}\label{9} Let $\cl A, \cl B$ be $C^*$-algebras and $F\subseteq \cl B$ be a finite set such that 
$$\cl B=C^*(\cl A\cup F).$$
If $\cl A$ satisfies SP then  under EPH.2 hypothesis so does $\cl B.$
\end{theorem}
\begin{proof} Assume that $$\pi: \cl B\to \cl B^{**}\subseteq B(H)$$ 
is an injective homomorphism. We denote 
$$\Omega_2=\cl B^{**}, \;\;\Omega_1=\overline{\pi(\cl A)}^{w^*}.$$
We claim that $\Omega_1$ satisfies WSP.
Indeed, if $u: \Omega_1\to B(K)$ is a bounded $w^*-$continuous homomorphism then 
$v_0=u|_{\pi(\cl A)}\to B(K)$ is completely bounded homomorphism. Thus $v_0$  extends to a completely bounded map
$$v_0: \pi(\cl B)\to B(K). $$  Therefore $v_0$ extends to a $w^*$-continuous completely bounded map 
$$v: \Omega_2\to B(K).$$
Since $u=v|_{\Omega_1}, u$  is completely bounded. 

We have that 
$$\Omega_2'=\Omega_1'\cap W^*(F)'$$
There exist projections $P_1,...,P_n\in M_2(W^*(F))$ generating $M_2(W^*(F))$ as von Neumann algebra.We have that,  
$$(M_2(\Omega_2))'\bar \otimes B(l^2)=\cap_{i=1}^n((M_2(\Omega_1))'\bar \otimes B(l^2))\cap \{(P_i\otimes I\}'.$$

Since $M_2(\Omega_1)$ satisfies WSP then $M_2(\Omega_1))'\bar \otimes B(l^2)$ is hyperreflexive. 

Lemma \ref{xxx}, implies using induction that $M_2(\Omega_2))'\bar \otimes B(l^2)$ is hyperreflexive. Thus $M_2(\Omega_2))$ and hence $\cl B^{**}\cong \Omega_2$  satisfies WSP. This implies that $\cl B$ satisfies SP.
\end{proof}

\begin{theorem}\label{10} Let $\cl A, \cl B$ be von Neumann algebras acting on the Hilbert space $H$ and $F\subseteq \cl B$ be a finite set such that 
$$\cl B=W^*(\cl A \cup F).$$
If $\cl A$ satisfies WSP then  under the EPH.2 hypothesis  so does $\cl B.$
\end{theorem}
\begin{proof}

We have that 
$$\cl B'=\cl A'\cap W^*(F)'$$
There exist projections $P_1,...,P_n\in M_2(W^*(F))$ generating $M_2(W^*(F))$ as von Neumann algebra.|We have that, 
$$M_2(\cl B)'\bar \otimes B(l^2)=\cap_{i=1}^n((M_2(\cl A)'\bar \otimes B(l^2))\cap \{P_i\otimes I\})'.$$
Since $M_2(\cl A)$ satisfies WSP then $M_2(\cl A))'\bar \otimes B(l^2)$ is hyperreflexive.
Lemma  \ref{xxx}, implies using induction on $n$ that $M_2(\cl B))'\bar \otimes B(l^2)$ is hyperreflexive. Thus $M_2(\cl B))$ and hence $\cl B$  satisfies WSP.
\end{proof}

\begin{theorem}\label{lmerd} Let $\cl A$ be a von Neumann algebra satisfying WSP and $\cl B=l^\infty(\bb N,\cl A)$ the algebra of bounded functions from $\bb N$ to $\cl A.$ Then $\cl B$ satisfies WSP.
\end{theorem}
\begin{proof} Let $u: \cl B\to B(H)$ be a weak* continuous bounded homomorphism. It suffices to prove that $u$ is completely bounded. 
Define, $$u_i: \cl A\to B(H), X\to u(e_i\otimes X).$$
Clearly $u_i$ is weak* continuous bounded homomorphism. If $P_i=u(e_i\otimes I_{\cl A})$ then 
$$u_i(X)=P_iu_i(X)P_i, P_i^2=P_i.$$
Since $\cl A$ satisfies WSP, there exist normal $*-$homomorphisms $\pi_i: \cl A\to B(H)$ and invertible operators $S_i$ such that $$\|S_i^{-1}\|\|S_i\|=\|u_i\|_{cb},\;\;u_i(X)=S_i^{-1}\pi_i(X)S_i,\;;\forall i\in \bb N.$$
 By \cite[Theorem 4.1]{lm} there exist $K>0, d>0$ such that 
$$\|u_i\|_{cb}\le K\|u_i\|^d\le K\|u\|^d, \forall i.$$
Thus $$\sup_i\|S_i^{-1}\|\|S_i\|<\infty.$$
Observe that if $Q_i=S_iP_iS_i^{-1}$ then 
$$Q_i^2=Q_i=Q_i^*,\;\; Q_iQ_j=0,i\neq j,\;\; \pi_i(X)=Q_i\pi_i(X)Q_i.$$
Define $\pi=\oplus_i\pi_i.$ This is a $*-$homomorphism. 

If $f\in \cl B$ then 
$$f(i)=X_i,\;\;\sup_i\|X_i\|<\infty\Rightarrow f=w^*-\sum_ie_i\otimes X_i.$$
We have that 
$$u(f)=w^*-\sum_iu(e_i\otimes X_i)=w^*-\sum _iu_i(X_i)=w^*-\sum_i S_i^{-1}\pi_i(X_i)S_i.$$
Fix $(f_{k,l})_{k,l}\in \Ball(M_n(\cl B)),$ where $f_{k,l}(i)=X_i^{k,l}$  

We have that
$$\|(u(f_{k,l})_{k,l}\|_{M_n(B(H))}=\|(\sum_i S_i^{-1}\pi_i(X_i^{k,l})S_i)_{k,l}\|_{M_n(B(H))}\le$$$$ \sup_i\|S_i^{-1}\|\|S_i\|
\|(\sum_i\oplus \pi_i(X_i^{k,l}))\|_{M_n(B(H))}=$$$$\sup_i\|S_i^{-1}\|\|S_i\|\|((\pi(f_{k,l})_{k,l})\|_{M_n(B(H))}\le \sup_i\|S_i^{-1}\|\|S_i\|<\infty.$$ 
Thus $u$ is completely bounded.

\end{proof}

The following can be deduced as a corollary of the previous theorem, but has a somewhat simpler direct proof. 

\begin{theorem}\label{pis} Let $\cl A$ be a weakly finitely generated von Neumann algebra and $\cl B=l^\infty(\bb N,\cl A)$ the algebra of bounded functions from $\bb N$ to $\cl A.$ Then $\cl B$ is also  weakly finitely generated and thus  under EPH.3   hypothesis  satisfies WSP.
\end{theorem}
\begin{proof}
The algebra $\cl B$ is isomorphic with the von Neumann algebra $l^\infty(N)\bar \otimes A.$ The masa $l^\infty(N)=W^*(S)$ is generated by a   selfadjoint element $S.$ Assume that $$\cl A=W^*(X_1,...,X_k).$$

We are going to prove that  $$(l^\infty(N)\bar \otimes A)=W^*(S\otimes I, I\otimes X_1,...,I\otimes X_k).$$

Let $ X\in \l^\infty(N), Y\in A,$ it suffices to prove that $$ X\otimes Y\in W^*(S\otimes I, I\otimes X_1,...,I\otimes X_k).$$

We have that,

$$X\otimes Y =(X\otimes I)(I\otimes Y)\in (W^*(S)\otimes I)( I\otimes W^*(X_1,...,X_k))=$$
$$(W^*(S\otimes I))( W^*(I\otimes X_1,...,I\otimes X_k))\subseteq $$
$$W^*(S\otimes I, I\otimes X_1,...,I\otimes X_k).$$

\end{proof}


\end{document}